\documentclass[11pt]{article}
\usepackage{amsmath,amsthm,amssymb,hyperref}
\newcommand{\remove}[1]{}
\newcommand{\ignore}[1]{}
\newtheorem{thm}{Theorem}[section]
\newtheorem{claim}[thm]{Claim}
\newtheorem{lem}[thm]{Lemma}
\newtheorem{define}[thm]{Definition}
\newtheorem{cor}[thm]{Corollary}

\newtheorem{example}[thm]{Example}

\newtheorem{THM}{Theorem}

\def\F{{\mathbb{F}}}

\def\R{{\mathbb{R}}}

\def\C{{\mathbb{C}}}

\newcommand{\ip}[2]{\langle #1,#2 \rangle}

\def\_{\,\,\,\,\,}

\def\span{\textsf{span}}

\def\arc{\textsf{circ}}
\def\line{\textsf{line}}
\def\dist{\textsf{dist}}

\newcommand{\eps}{\epsilon}

\begin{document}

\title{Sylvester-Gallai type theorems for approximate collinearity}
\author{Albert Ai \thanks{Department of of Mathematics, Princeton University, Princeton NJ.
Email: \texttt{aai@princeton.edu}.} \and Zeev Dvir \thanks{Department of Computer Science and Department of Mathematics, Princeton University, Princeton NJ.
Email: \texttt{zeev.dvir@gmail.com}. Research partially
supported by NSF grant CCF-0832797.} \and Shubhangi Saraf \thanks{Department of Computer Science and Department of Mathematics, Rutgers University.
Email: \texttt{shubhangi.saraf@gmail.com}.} \and Avi Wigderson \thanks{School of Mathematics, Institute for Advanced Study.
Email: \texttt{avi@ias.edu}.}}
\date{}
\maketitle

\begin{abstract}

We study questions in incidence geometry where the precise position of points is `blurry' (e.g. due to noise, inaccuracy or error). Thus lines are replaced by narrow tubes, and more generally affine subspaces are replaced by their small neighborhood. We show that the presence of a sufficiently large number of approximately collinear triples in a set of points in $\C^d$ implies that the points are {\em close} to a low dimensional affine subspace. This can be viewed as a stable variant of the Sylvester-Gallai theorem and its extensions.

Building on  the recently found connection between Sylvester-Gallai type theorems and complex  Locally Correctable Codes (LCCs), we define the new notion of {\em stable} LCCs, in which the (local) correction procedure can also handle small perturbations in the euclidean metric. We prove that such stable codes with constant query complexity do not exist. No impossibility results were known in any such local setting for more than 2 queries.
\end{abstract}

%\newpage 
%\pagenumbering{arabic}

%%%%%%%%%%%%%%%%%%%%%%%%%%%%%%%%%%%%%%%%%%%%%%%%%%%%%%%%%%%%
\section{Introduction}
%%%%%%%%%%%%%%%%%%%%%%%%%%%%%%%%%%%%%%%%%%%%%%%%%%%%%%%%%%%%

The Sylvester-Gallai theorem is a statement about configurations of points in $\R^d$ in which there is a certain structure of  collinear triples.

\begin{thm}[Sylvester-Gallai]
Suppose $v_1,\ldots,v_n \in \R^d$ are such that for all $i \neq j \in [n]$ there is some $k \in [n] \setminus\{i,j\}$ for which $v_i,v_j,v_k$ are on a line. Then all the points $v_1,\ldots,v_n$ are on a single line. 
\end{thm} 
 
This theorem takes {\em local} information about dependencies between points and concludes {\em global} information about the entire configuration. For more on the history and generalizations of this theorem we refer to the survey \cite{BM90}. A complex variant of this theorem was proved by Kelly:
 
\begin{thm}[\cite{Kel86}]
Suppose $v_1,\ldots,v_n \in \C^d$ are such that for all $i \neq j \in [n]$ there is some $k \in [n] \setminus\{i,j\}$ for which $v_i,v_j,v_k$ are on a (complex) line. Then all the points $v_1,\ldots,v_n$ lie on a single (complex) plane. 
\end{thm}

The global dimension bound given by Kelly's theorem is tight since, over the complex numbers, there are two-dimensional configurations of points satisfying the condition on triples. 

In a recent work, Barak et. al. \cite{BDWY11} proved quantitative (or fractional) analogs of Kelly's theorem in which the condition `for all $i \neq j \in [n]$' is relaxed and we have information only on a large subset of the pairs of points for which there exists a third collinear point\footnote{The sets of points satisfying the conditions of the theorem were called $\delta$-SG configurations in \cite{BDWY11}.}.
 
\begin{thm}[\cite{BDWY11}]\label{thm-bdwy}
Suppose $v_1,\ldots,v_n \in \C^d$ are such that for all $i \in [n]$ there exist at least $\delta (n-1)$ values of $j \in [n]\setminus \{i\}$ for which there is $k \in [n] \setminus\{i,j\}$ such that $v_i,v_j,v_k$ are on a  line. Then all the points $v_1,\ldots,v_n$ lie in an affine subspace of dimension $13/\delta^2$.
\end{thm}

A more recent work \cite{DSW12} improves the dimension upper bound obtained in the above theorem from  $O(1/\delta^2)$ to the asymptotically tight $O(1/\delta)$  and also gives a new proof of Kelly's theorem (when $\delta=1$ one gets an upper bound of 2 on the dimension). 

In this work we consider configurations of points in which there are many triples that are `almost' collinear, in the sense that there is a line close to all three points (in the usual Euclidean metric on $\C^d$). Equivalently, the points are contained in a narrow tube. Our goal is to prove stable analogs of the above theorems, where stable means that the conclusion of the theorem will not change significantly when perturbing the point set slightly. Clearly, in such settings one can only hope to prove that there is a low dimensional subspace that {\em approximates} the set of points. There are many technical issues to discuss when defining approximate collinearity and there are some non trivial examples showing that word-to-word generalizations of the above theorems do not hold in the approximate-collinearity setting (at least for some of the possible definitions). Nonetheless, we are able to prove several theorems of this flavor for configurations of points satisfying certain `niceness' conditions. We also study stable variants of error correcting codes (over the reals) which are locally correctable, in which such approximately collinear tuples of points naturally arise from the correcting procedure. 

In \cite{BDWY11}, a connection was made between the Sylvester-Gallai theorem to a special kind of error correcting codes called Locally Correctable Codes (LCCs). In these codes, a receiver of a corrupted codeword can recover a single symbol of the codeword correctly, making only a small number of queries to the corrupted word.  When studying linear LCCs over the real or complex numbers one encounters the same type of difficulties in trying to convert local dependencies into global dimension bounds. Building on this connection, and our ability to analyze `approximate' linear dependencies, we define the notion of {\em stable} LCC and show that these do not exist for constant query complexity. Stable LCCs correspond to configurations of points with many approximately dependent small subsets and so our techniques can be used to analyze them.

We note here that understanding the possible intersection structure of tubes in high dimensional real space comes up in connection to other geometric problems, most notably the Euclidean Kakeya problem \cite{Tao01} (we do not, however, see a direct connection between our results and this difficult problem).

Our proof techniques extend those of \cite{BDWY11,DSW12} and rely on high rank properties of sparse matrices whose support is a `design'. In this work we go a step further and, instead of relying on rank alone, we need to bound the number of small singular values of such matrices.

\paragraph{Organization:} In Section~\ref{sec-statement} we formally state our results for point configurations. The results are stated in several sub-sections, corresponding to different variants of the problem we consider. In Section~\ref{sec-resultslcc} we define stable LCCs and state our results in this scenario. The proofs are given in Sections~\ref{sec-pfgen} -- \ref{sec-apxlcc}. 

\paragraph{Notations:}  We use big `O' notation to suppress absolute constants only.  For two complex vectors $u,v \in \C^d$ we denote their inner product by $\ip{u}{v} = \sum_{i=1}^d u_i \cdot \overline{v_i}$ and use $\|v\| = \sqrt{\ip{v}{v}}$ to denote the $\ell_2$ norm. For an $m \times n $ matrix $A$, we denote by $\|A\|$ the norm of $A$ as a vector of length $mn$ (i.e., the Forbenius norm).  The {\em distance} between two points $u,v \in \C^d$ is defined to be $\|u-v\|$ and is  denoted $\dist(u,v)$. For a set $S \subset \C^d$ and a point $v \in \C^d$ we denote $\dist(v,S) = \inf_{u \in S} \dist(u,v)$. We let $S^d \subset \C^{d+1}$ denote the $d$-dimensional unit sphere in complex $d+1$ dimensional space. By fixing a basis we can identify each $v \in S^d$ with a $d+1$ length complex vector of $\ell_2$-norm equal to one. 

%%%%%%%%%%%%%%%%%%%%%%%%%%%%%%%%%%%%%%%%%%%%%%%%%%%%%%%%%%%%
\section{Point configurations}\label{sec-statement}
%%%%%%%%%%%%%%%%%%%%%%%%%%%%%%%%%%%%%%%%%%%%%%%%%%%%%%%%%%%%

In this section we state our results concerning point configurations. The first section, Section~\ref{sec-resultsaffine} deals with the most natural setting -- the affine setting -- in which we consider sets of points in $\C^d$ with many almost-collinear triples. In Section~\ref{sec-resultsproj} we consider the projective setting where the points are located on the sphere and collinearity is replaced with linear dependence. Section~\ref{sec-resultsgen} states a more general theorem from which both the affine and the projective results follow.

\subsection{The affine setting}\label{sec-resultsaffine}

We begin with the  definition of an $\eps$-line.

\begin{define}[$\line$,$\line_\eps$]\label{def-line}
Let $u \neq v \in \C^d$. We define $\line(u,v) = \{ \alpha u + (1-\alpha) v \,|\, \alpha \in \C\}$ to be the complex line passing through $u,v$. We define $\line_\eps(u,v) = \{ w\in \C^d \,|\, \dist(w,\line(u,v)) \leq \eps \}$.
\end{define}

The following definition will be used to replace the notion of dimension with a more stable definition.

\begin{define}[\bf $dim_\eps$]
 For a set of points $V \subset \C^d$ and $\eps>0$ we denote by $\dim_\eps(V)$ to be the minimal $k$ such that there exists a $k$-dimensional subspace\footnote{The difference of 1 between affine and linear dimension will not be significant in this paper and so we use a linear subspace in the definition.} $L \subset \C^d$ such that $\dist(v,L) \leq \eps$ for all $v \in V$.
\end{define}

To give an idea of the subtleties that arise when dealing with approximate collinearity, take an orthonormal basis $e_1,\ldots,e_d$ in $\C^d$  and consider the set $V = \{e_1,e_1',\ldots,e_d,e_d'\}$ with $e_i' = (1+\eps)e_i$. Clearly, there is no low dimensional subspace that approximates this set of points, even though there are many pairs for which there is a third $\eps$-collinear point ($e_i'$ is $\eps$-close to the line passing through $e_i$ and any other third point). An obvious solution to this problem is to require that the minimal distance between each pair of points is bounded from below (say by 1), so that the condition of $\eps$-collinearity is meaningful. We now describe another, less trivial, example which shows that this condition alone is not sufficient in general.

\begin{example}\label{ex-affine}  Let $e_1,\ldots,e_d$ be an orthonormal basis in $\C^d$. Let $v_i = Be_i$, $u_i = (B-1)e_i$  for all $i \in [d]$ and let $V = \{e_i,u_i,v_i \,|\, i \in [d] \}$ be a set of $n = 3d$ points. Then for all $i,j \in [d]$ we have $u_i \in \line_\eps(v_i,e_j)$ and $v_i \in \line_\eps(u_i,e_j)$ with $\eps = 1/B$. Thus, there are many $\eps$-collinear triples in $V$ (as in the conditions of Theorem~\ref{thm-apxsgaffine} with $\delta = 1/3$). However, for any  subspace $L$ of dimension $o(n)$, the distance of at least one of the point $v_i$ to $L$ must be at least $\Omega(B)$ (this can be shown, e.g., using Lemma~\ref{lem-pert}). 
\end{example}

In this example, we had $\eps= 1/B$, where $B$ is roughly equal to the ratio between the smallest and the largest distance, or the `aspect ratio' of $V$. We will  prevent this scenario by requiring that $\eps$ will be sufficiently smaller than $1/B$, where $B$ will be the aspect ratio. This motivates the following definition.

\begin{define}[$B$-balanced]\label{def-balanced}
A set $V \subset \C^d$ is said to be $B$-balanced if  $1 \leq \dist(v,v') \leq B$ for all $v \neq v' \in V$.
\end{define}

The following theorem gives the most easy to state version of our results.

\begin{THM}\label{thm-apxsgaffinesimple}
Let $n, d > 0$ be integers and let $B, \eps > 0$ be real numbers with 	$ \eps < 1/16B$. Let $V = \{v_1,\ldots,v_n\} \subset \C^d$ be $B$-balanced and suppose that for every $i \neq j \in [n]$ there exists $k \in [n]\setminus \{i,j\}$ such that $v_k \in \line_\eps(v_i,v_j)$. Then, $\dim_{\eps'}(V) \leq O(B^{6})$  with $\eps' \leq O( \eps B^{2.5})$.	
\end{THM}

Observe that a corollary of this theorem is that the number of points, $n$, is bounded from above by a function of $B$. A priori, we did not have this bound since a $B$-balanced configuration in $\C^d$ can have an unbounded number of points when $d$ grows.

Notice that our definition of $\eps$-collinearity is not symmetric in that it depends on the order of the triple. As is shown in Lemma~\ref{lem-arcreverseaffine}, this is not an issue for $B$-balanced configurations, as long as we are willing to replace $\eps$ with $\eps B$. For general (i.e., non balanced) configurations the situation can be more complicated and it is possible that using a stronger collinearity condition (e.g., requiring that any permutation of the triple satisfies our condition) is sufficient for obtaining a global dimension bound. 

Theorem~\ref{thm-apxsgaffinesimple} will be a special case of the following, more general theorem, in which we only have information of a subset of the pairs $(i,j)$. Assuming $V$ has many $\eps$-collinear triples (for each point), we derive an upper bound on  $\dim_{\eps'}(V)$ for $\eps'$ which depends on the other parameters. We also derive a  better bound on $\eps'$ when restricting to a subset of the points. 

\begin{THM}\label{thm-apxsgaffine}
Let $n, d > 0$ be integers.  Let $B, \delta, \eps > 0$ be real numbers with 	$ \eps < 1/16B$. Let $V = \{v_1,\ldots,v_n\} \subset \C^d$ be $B$-balanced and suppose that for every $i\in [n]$ there are at least $\delta(n-1)$ values of $j \in [n]\setminus \{i\}$ for which there exists $k \in [n]\setminus \{i,j\}$ such that $v_k \in \line_\eps(v_i,v_j)$. Then
\begin{enumerate}
	\item $\dim_{\eps'}(V) \leq O(B^{6}/\delta^2)$  with $\eps' \leq O( \eps B^{2.5}/\delta^{0.5})$.
	\item There exists a subset $V' \subset V$ of size $\Omega(n)$ with $\dim_{\eps''}(V') \leq O(B^{6}/\delta^2)$ and $\eps'' \leq O(B\eps)$.
\end{enumerate}
\end{THM} 

In both of the above theorems, the parameter $B$  appears in the resulting global dimension bound. We suspect that this dependence can be removed so that the bound on the dimension will be $O(1)$ in Theorem~\ref{thm-apxsgaffinesimple} and $O(1/\delta^2)$ (or even $O(1/\delta)$) in Theorem~\ref{thm-apxsgaffine}. The blowup in $\eps'$, compared to $\eps$ is also likely to be suboptimal.

% Regarding the blowup in $\eps'$ (compared to $\eps$) Example~\ref{ex-affine}  shows that a factor of at least $\Omega(B^2)$ is unavoidable in Theorem~\ref{thm-apxsgaffine} (at least when $\delta \leq 1/3$). It also shows that the factor of $\Omega(B)$ must appear in $\eps''$, the bound on the distances in $V'$ (when we only care about a constant fraction of the points in $V$). It is likely that the factor of $B^{2.5}$ in our proof could be further improved to $B^2$.

A stronger definition of collinearity, for which Example~\ref{ex-affine} fails, is to require that each point in the triple is $\eps$-close to the line spanned by the other two points. Let us call such triples {\em strongly} $\eps$-collinear triples. It is easy to see that, in Example~\ref{ex-affine}, the triples do not satisfy this stronger definition. Thus, it is possible that one could prove analogs of Theorem~\ref{thm-apxsgaffine} for configurations that are not $B$-balanced using this stronger definition of approximate collinearity. 

We conclude this discussion with yet another example showing that, even for the case $\delta=1$ (i.e, the original Sylvester-Gallai condition) the weak definition of $\eps$-collinearity requires some balancedness condition (though potentially weaker).

\begin{example}
Fix some large $B > 0$. Take an orthonormal basis $e_1,\ldots,e_d \in \C^d$ and define $V = \{0\} \cup \bigcup_{i\in [d]}\left\{ B^{i-1}e_i, (B^{i-1}+1)e_i \right\}$. One can verify by induction that for every $u,v \in V$ there is a third point inside $\line_{\eps}(u,v)$ with $\eps \approx 1/B$.  There is also no low dimensional subspace that approximates $V$ (similar to the previous examples). 
\end{example}

\subsection{The projective setting}\label{sec-resultsproj}

Since the definition of $\eps$-collinearity (that is, $v_k \in \line_\eps(v_i,v_j)$) is sensitive to scaling, a projective statement of Theorem~\ref{thm-apxsgaffine}, in which these scaling issues do not arise, seems natural. In this setting  we consider points on a sphere and lines are replaced by circles (two dimensional subspaces intersected with $S^d$).

\begin{define}[$\arc$,$\arc_\eps$]\label{def-arc}
Let $u,v \in S^d$. We define $\arc(u,v) = \span\{u,v\} \cap S^d$. We define $\arc_\eps(u,v) = \{ w\in S^d \,|\, \dist(w,\arc(u,v)) \leq \eps \}$.
\end{define}

An instructive example  in the projective case is the following:

\begin{example}\label{ex-proj}
Take $V$ to be a maximal set in $S^d$ with pairwise distances at least $\mu >0$ (so that $n \approx (c/\mu)^{d}$ with $c$ a constant). Since every point in $S^d$ is of distance at most $\mu$ from one of the points in $V$ (otherwise we could add it) we get that each set $\arc_\mu(v_i,v_j)$ contains at least $\Omega(1/\mu) > 2$ points from $V$. On the other hand, for any low dimensional subspace $L$ (say, with dimension $d'$ independent of $n$) almost all points in $V$ will have distance at least $1/100$ from $L$.
\end{example}

From this example we see that there needs to be some upper bound on $\eps$ as a function of the minimal distance in the set. We will use the following definition to replace $B$-balancedness.

\begin{define}[$\mu$-separated]\label{def-sep}
A set $V \subset S^d$ is said to be $\mu$-separated if for every $u \neq v \in V$ we have $\min\{ \dist(u,v), \dist(u,-v)\} \geq \mu$. 
\end{define}

We now state our theorem for points on a sphere. 

\begin{THM}\label{thm-apxsg}
Let $n, d > 0$ be integers and let $\delta,\mu,\eps > 0$ be real numbers with $ \eps < \mu^2/32$. Let $V = \{v_1,\ldots,v_n\} \subset S^d$ be $\mu$-separated and suppose that for every $i\in [n]$ there are at least $\delta(n-1)$ values of $j \in [n]\setminus \{i\}$ for which there exists $k \in [n]\setminus \{i,j\}$ such that $v_k \in \arc_\eps(v_i,v_j)$. Then
\begin{enumerate}
	\item $\dim_{\eps'}(V) \leq O(1/\delta^2\mu^6)$ with $\eps' \leq O( \eps/\delta^{0.5}\mu^{2.5})$.
	\item There exists a subset $V' \subset V$ of size $\Omega(n)$ with $\dim_{\eps''}(V') \leq O(1/\delta^2\mu^6)$ and $\eps'' \leq O(\eps/\mu)$.
\end{enumerate}
\end{THM}

Notice that, when compared with Theorem~\ref{thm-apxsg}, the parameters $\mu$ corresponds to $1/B$. However, the condition on $\eps < \mu^2/32$ is more restrictive in this case. We do not know whether this condition can be improved to $\eps  \leq O(\mu)$. 
As is the case with Theorem~\ref{thm-apxsg}, we do not expect the dependency in the dimension bound and in $\eps'$ to be tight.

\subsection{The general statement}\label{sec-resultsgen}

Both  Theorem~\ref{thm-apxsgaffine} and Theorem~\ref{thm-apxsg} will follow from a more general statement requiring a set of points with a family of $\eps$-dependent triples satisfying certain  conditions. 

\begin{define}[$(\eps,\mu)$-dependent]
We say that a triple of points $u,v,w \in \C^d$ is $(\eps,\mu)$-dependent if there exists complex numbers $\alpha,\beta,\gamma$ with $|\alpha|,|\beta|,|\gamma| \in [\mu,1]$ such that $$\| \alpha u + \beta v + \gamma w \| \leq \eps.$$
\end{define}

\begin{define}[$(p,g)$-design]
Let $T \subset {[n] \choose 3}$ be a family of triples in $[n]$. We say that $T$ is a $(p,g)$-design if 
\begin{enumerate}
	\item For all $i \in [n]$ there are at least $p$ triples in $T$ that contain $i$.
	\item For all $i \neq j \in [n]$ there are at most $g$ triples in $T$
 containing both $i$ and $j$.
\end{enumerate}
\end{define}

The following theorem gives a low dimensional subspace that approximates all points in a configuration in which there is a design of triples that are $(\eps,\mu)$-dependent. Below we will also prove a slightly more refined statement (see Theorem~\ref{thm-apxsggenav}) giving better distance from $L$ for {\em many} points in the configuration.

\begin{THM}\label{thm-apxsggen}
Let $n, d > 0$ be integers and $p,g,\delta,\mu,\eps > 0$ be real numbers. Let $V = \{v_1,\ldots,v_n\} \subset \C^d$, $T \subset {[n] \choose 3}$ be such that $T$ is  $(p,g)$-design, and for every  $\{i,j,k\} \in T$ the triple $v_i,v_j,v_k$ is $(\eps,\mu)$-dependent. Then, $$\dim_{\eps'}(V) \leq \frac{2n^2 g^2}{p^2 \mu^4}$$ with $$\eps' \leq \frac{5\eps \sqrt{g|T|}}{p\mu^2}.$$
\end{THM}

A setting of the parameters which will be most relevant to us is when $|T|$ is quadratic in $n$, $p$ is linear in $n$ and $g$ and $\mu$ are constants. In this case we get a constant upper bound on the dimension $\dim_{\eps'}(V)$ with $\eps' = O(\eps)$.
 
The proof of Theorem~\ref{thm-apxsggen} is given in the next section with the proofs of Theorems~\ref{thm-apxsgaffine} and ~\ref{thm-apxsg} in Sections~\ref{sec-pfaffine} and \ref{sec-pfproj} respectively. We give a high level overview of the proof below.

\paragraph{Proof overview:}

We place the points $v_1,\ldots,v_n$ as rows in a matrix $A$. We then  use the triple family $T$ to construct a matrix $M$ such that
\begin{itemize}
	\item $M$ is a $|T| \times n$ matrix whose support is determined by $T$. More precisely, the non zero coordinates of the $t$'th row of $M$, with $t \in T$, will be the three elements in $t$. 
	\item The values of the entries of $M$ will be in absolute value between $\mu$ and $1$.
	\item The product $M\cdot A$ will have small Forbenius norm.
\end{itemize}

We then observe that the matrix $X = M^{*} M$ is diagonal dominant (its diagonal elements are much larger than its off-diagonal elements). This implies, using the Hoffman-Wielandt inequality,  that $M$ has only a few small singular values. From this we get that the columns of $A$ must have small distance (on average) to the span of the small singular vectors of $M$ and so can be approximated well by a low dimensional space. We then show that the same statement holds when one replaces the columns of $A$ with the rows of $A$ (a fact which generalizes the simple fact that the row rank is equal to the column rank). Using the bound on the average distance of rows we argue that there is a large subset that is approximated well by a low dimensional subspace. We then extend this to {\em all} points using interpolation.

\section{Stable Locally Correctable Codes}\label{sec-resultslcc}

% Local versions of error correcting codes over finite fields, e.g. locally testable, locally decodable and locally correctable codes, have been studied for a variety of motivations and applications (e.g. \cite{??}). Here we study such local analogs over the Real and Complex numbers.
 
Before discussing local correction, we briefly mention the exciting recent developments regarding `standard' (non-local) error correcting codes over the reals. Like in the analogous theory over finite fields, one would like to encode (typically via a linear transformation) a vector of entries from a given field $\F$ by a longer one, such that the original message can be decoded even when  some entries of the codeword are corrupted. The breakthrough of `compressed sensing' by Donoho  and Candes-Tao, and subsequent developments (see e.g. \cite{CT05, RV05, Donoho06, KT07, DMT07, GLW09}) has lead to an understanding of codes over the reals that is almost as good as in the finite-field case. In particular, there are real-valued codes which achieve the gold-standard of coding theory of constant rate linear codes with efficient encoding and decoding algorithms from a linear number of errors of arbitrary magnitude. Moreover, these codes have {\em stable} versions which can recover a vector close to the original message even if small errors affect {\em all} coordinates of the encoding. Our local variant may be viewed as one local analog of such stable codes.

Informally, Locally Correctable Codes (LCCs) are error correcting codes that allow the transmission of information over a noisy channel so that the symbols of the transmitted words have many local dependencies between them. The most general definition requires that one can reconstruct (w.h.p) {\em any} coordinate in a possibly corrupted codeword, using a small number of (randomly chosen) queries to the other coordinates. The noise model is adversarial, meaning that the corrupted positions are arbitrary (and not random) and one only has a bound on the total number of errors (which is usually assumed to be a small constant fraction). LCCs are closely related to another type of codes - Locally Decodable Codes (LDCs)-- whose study was initiated in a work of Katz and Trevisan \cite{KTldc}. We refer the interested reader to \cite{Y_now} for the relevant background on LDCs and LCCs and their applications in computer science.
 
The connection between LCCs and the Sylvester-Gallai theorem was first observed in  \cite{BDWY11}. When studying the special case of {\em linear} LCCs (i.e., LCCs that are given by linear mappings over a field) one can easily show that LCCs are equivalent to point configurations with many linearly dependent small subsets. The general definition of linear LCCs is as follows (we fix the field to be $\C$ but the same definition works for any field). We use $w(v)$ to denote the number of non zero elements in a vector $v \in \C^n$.
 
% \begin{define}\label{def-lcc}
% A $(q,\delta)-LCC$  over a field $\F$ is a set\footnote{One can consider, more generally, multi-sets.} of vectors $V = \{v_1,\ldots,v_n\} \subset \F^d$, spanning $\F^d$, such that for each $i \in [n]$ and every set $S \subset [n]$ of size at most $\delta n$ there exists a set $J \subset [n]\setminus S$ with $|J|\leq q$ such that $v_i \in \span\{v_j \,|\, j \in J\}$. The parameter $q$ is called the {\em query complexity} of $V$ and the parameter $d$ is the {\em dimension} of $V$.
% \end{define}

\begin{define}[Linear LCC -- first definition]\label{def-lcc1}
A $(q,\delta)$-LCC over $\C$ is a linear subspace $U \subset \C^m$ such that there exists a randomized
decoding procedure $D : \C^m \times [m] \mapsto \C$ with the following properties:
\begin{enumerate}
\item For all $x \in U$, for all $i \in [m]$ and for all $v \in \C^m$ with $w(v) \leq \delta m$ we have that $D\left( x + v, i\right) = x_i$ with probability at least $3/4$ (the probability is taken only over the internal randomness of $D$).
\item For every $y \in \C^m$ and $i \in [m]$, the decoder $D(y,i)$ reads at most $q$ positions in $y$.
\end{enumerate}
The {\em dimension} of an LCC is simply its dimension as a subspace of $\C^m$. 
\end{define}

It is shown in \cite{BDWY11} that, w.l.o.g. the decoding procedure is {\em linear}, in the sense that it first picks a set of at most $q$ coordinates to read and then outputs a linear combination of them (with coefficients in $\C$). This linearity of the decoder implies that, for each coordinate in the code, there are many small subsets of the other coordinates that span it. Since each coordinate corresponds to a row of the generating matrix of the code, we obtain a configuration of points with many dependent small subsets. We will make this formal in the next definition, which is equivalent to the first definition, if one replaces $\delta$ with the slightly worse bound of $\delta/q$ (when $q$ is constant this change is negligible).

\begin{define}[Linear LCC -- second definition]\label{def-lcc2}
We say that a finite set $V = \{v_1,\ldots,v_n\} \subset \C^d$ is a $(q,\delta)$-{\em LCC} if for every $i \in [n]$ and every set $S \subset [n]$ of size $|S| \leq \delta n$ there exists a set $J \subset [n]\setminus S$ with $|J| \leq q$  such that $v_i \in \span(v_j\,|\, j \in J)$.
\end{define}

The main open problem regarding LCCs is to determine the maximum dimension (as a function of $n$) when we fix $q,\delta$ to be constants. Intuitively, the larger $d$ is, the more `information' we can transmit using the code (the {\em rate} of the code if $d/n$). While the case of $q=2$ is understood quite well ($d$ is at most logarithmic over finite fields and constant over characteristic zero \cite{BDWY11,BDSS11}), it is an open problem to determine the maximum dimension of a $q$-query LCC when $q > 2$. There are exponential gaps between the known lower and upper bound. For example, when $q=3$, the best upper bound is $d \leq O(\sqrt{n})$ \cite{Woodruff, KdW} while the best constructions give poly-logarithmic $d$ over finite fields and constant $d$ over characteristic zero. We refer the reader to the survey article \cite{Dvir-survey} for more background on LCCs and for an overview of the known constructions.
 
Due to their roots in coding theory, LCCs were traditionally studied exclusively over finite fields. The study of LCCs over arbitrary fields was initiated in \cite{BDWY11} and was motivated by its connection to the Sylvester-Gallai theorem. Another motivation comes from a work connecting LCCs with an approach for constructing {\em rigid matrices}  over infinite fields \cite{Dvi10}. We note here that for $q>2$, the best upper bounds on the dimensions of LCCs are the same, no matter what the field is. This also motivates the study of LCC's over infinite fields as a potentially easier scenario to tackle first, before proceeding to codes over finite fields (where we have fewer techniques).

Our methods enable us to prove strong upper bounds on the dimension of codes that we call {\em stable LCCs}. Before discussing the relation between stable and non-stable LCCs we give the formal definition. 

\begin{define}[$\span_B$]
Let $v,u_1,\ldots,u_m \in \C^d$. We say that $v \in \span_B(u_1,\ldots,u_m)$ if there exist $a_1,\ldots,a_m \in \C$ with $|a_i| \leq B$ for all $i$ and $v = \sum_{i=1}^m a_i u_i$.
\end{define}

\begin{define}[Stable LCC]\label{def-stablelcc}
We say that a finite set $V = \{v_1,\ldots,v_n\} \subset \C^d$ is a $(q,\delta,B,\eps)$-{\em stable LCC} if for every $i \in [n]$ and every set $S \subset [n]$ of size $|S| \leq \delta n$ there exists a set $J \subset [n]\setminus S$ with $|J| \leq q$  such that $\dist(v_i, \span_B(v_j\,|\, j \in J)) \leq \eps$.
\end{define}
 
Notice that this definition is incomparable to  Definition~\ref{def-lcc2}: On the one hand, we restrict the linear dependencies to use only coefficients of bounded magnitude. On the other hand, we allow the linear combinations to result in an `approximate' vector, instead of the exact one. To see why the bound on the coefficients is natural (once you allow approximate recovery), notice that the decoder can handle small perturbations {\em even in the `correct positions'}. Stated in the scenario of Definition~\ref{def-lcc1}, suppose that in a received codeword at most $\delta$ fraction of the positions are completely changed (to arbitrary values) and, in addition, all other coordinates are perturbed by some small $\alpha$ in Euclidean distance. Then, the decoder can still recover (approximately) the value of a given codeword coordinate by reading at most $q$ other positions, as long as $\alpha \ll \eps/qB$. Since each of the read coordinates is  multiplied by a coefficient that can be as large as $B$ and the errors sum over $q$ positions, we get at most $\alpha\cdot qB$ resulting error in the output of the decoder.\footnote{One can potentially define stable LCCs in this sense (as in Definition~\ref{def-lcc1}) and then prove (similarly to \cite{BDWY11}) that, up to constants, it is equivalent to Definition~\ref{def-stablelcc} (we did not verify the details).}

The next simple claim shows that Definition~\ref{def-stablelcc} is also stable in the sense that, perturbing the elements in a stable LCC gives another stable LCC (with slightly worse parameters).

\begin{claim}
Let $V = \{v_1,\ldots,v_n\} \subset \C^d$  be a $(q,\delta,B,\eps)$-{\em stable LCC} and let $V = \{v'_1,\ldots,v'_n\} \subset \C^d$ be such that $\dist(v_i,v_i') \leq \alpha$ for all $i \in [n]$. Then $V'$ is a  $(q,\delta,B,\eps')$-{\em stable LCC} with $\eps' \leq \eps + (qB+1)\alpha$.
\end{claim}
\begin{proof}
Take some $v_i \in V$ and a set $J \subset [n]$ of size $|J| \leq q$ such that $\dist(v_i, \span_B(v_j\,|\, j \in J)) \leq \eps$. Then, there exist coefficients $b_j, j \in J$ with $|b_j| \leq B$ and such that $$ \left\| v_i - \sum_{j \in J}b_j v_j \right\| \leq \eps.$$ Replacing $v_i$ with $v_i'$ we get that $$ \left\|v_i' - \sum_{j \in J}b_j v_j' \right\| \leq \eps + \|v_i - v_i'\| + \sum_{j \in J}b_j\|v_j - v_j'\| \leq \eps + (qB+1)\alpha.$$ 
\end{proof}

Notice that, if we didn't have the bound on the coefficients in the span, the small perturbations would have resulted in  large errors in the linear combinations. Intuitively, if $u$ is not in $\span_B(u_1,\ldots,u_m)$ then a small perturbation to the $u_i$'s may result in $u$ being very far  from $\span(u_1,\ldots,u_m)$. This explains the need for two separate stability parameters, $\eps$  and $B$.

Our main result regarding stable LCC's is the following theorem:

\begin{THM}\label{thm-apxlcc}
Let $V = \{v_1,\ldots,v_n\} \subset \C^d$ be a $(q,\delta,B,\eps)$-{\em stable LCC}. Then, $$\dim_{\eps'}(V) \leq O((qB/\delta)^4)$$ with $$\eps' = O( q^2 B \eps / \delta^{1.5}).$$
\end{THM}

In particular, when $q$ is a constant and $B$ and $\delta$ are fixed, the upper bound on $\dim_{\eps'}$ can be interpreted as saying that there {\em do not exist} stable $q$-query LCCs, where `do not exist' means that the amount of information one can transmit is constant, regardless of the codeword length. The proof of Theorem~\ref{thm-apxlcc}, which follows the same lines as the proof of the Sylvester-Gallai type theorems, works also for the more general setting where $V$ is allowed to be an ordered multiset (i.e., when different $v_i$'s can repeat several times). 

If one sets $\eps=0$ the definition of stable LCC changes into a definition of an LCC with bounded coefficients. That is, the linear dependencies are required to be exact (as in the usual definition of an LCC) and, in addition, need to use bounded coefficients. Applying Theorem~\ref{thm-apxlcc} to this special case one gets $\eps'=0$ and so obtains the stronger conclusion that the set $V$ is actually {\em contained} in a low dimensional space. Stated more formally, we have:
\begin{cor}
Let $V = \{v_1,\ldots,v_n\} \subset \C^d$ be a $(q,\delta,B,0)$-{\em stable LCC}. Then, $$\dim(V) \leq O((qB/\delta)^4)$$.
\end{cor}

%%%%%%%%%%%%%%%%%%%%%%%%%%%%%%%%%%%%%%%%%%%%%%%%%%%%%%%%%%%%
\section{Proof of Theorem~\ref{thm-apxsggen}}\label{sec-pfgen}
%%%%%%%%%%%%%%%%%%%%%%%%%%%%%%%%%%%%%%%%%%%%%%%%%%%%%%%%%%%%

We will derive Theorem~\ref{thm-apxsggen} from the following, more refined, statement.

\begin{thm}\label{thm-apxsggenav}
Under the same conditions as in Theorem~\ref{thm-apxsggen}, there exists a subspace $L \subset \C^d$ with $$\dim(L) \leq \frac{2n^2 g^2}{p^2 \mu^4}$$ and such that 	$$ \sum_{i=1}^n \dist( v_i, L)^2 \leq \frac{4|T| \eps^2}{\mu^2 p}.$$ 
\end{thm}
\begin{proof}

First, observe that, for convenience, we can take $d=n$ so that the vectors $v_i$ are in $\C^n$. The case $d > n$ is not interesting since we can restrict our attention to the span of the $n$ vectors. The case $d < n$ can be similarly handled by padding each vector with zeros.  

Let $m = |T|$. We  use $T$ to construct an $m \times n$ matrix $M$  so that there is a one-to-one correspondence between rows of $M$ and  elements of $T$. By our assumptions, for each triple $t = \{i,j,k\} \in T$ there are complex numbers $\alpha,\beta,\gamma$ such that $\|\alpha v_i + \beta v_j + \gamma v_k\| \leq \eps$ and s.t $\mu \leq |\alpha|,|\beta|,|\gamma| \leq 1$. Let $s_t$ denote the row vector in $\C^n$ with the value $\alpha$ in position $i$, the value $\beta$ in position $j$, the value $\gamma$ in position $k$ and zeros everywhere else. We define $M$ to be the matrix with  rows  $s_t$ where $t$ goes over all triples in $T$ (in some order). 

Next, let $A$ be a complex $n \times n$ matrix whose $i$'th row is the vector $v_i$. Then, from our definition of the rows of $M$, we have that the rows of the $m \times n$ matrix 
\begin{equation}\label{eq-E=MA}
	E = M  A
\end{equation}
all have norm at most $\eps$. 

The next claim summarizes some of the properties of $M$ that we will use. All three items follow immediately from the fact that $T$ is a $(p,g)$-design and the bounds on the entries of $M$.
\begin{claim}\label{cla-propM}
Let $M$ be as above and let $M_j \in \C^m$, $j \in [n]$ denote the $j$'th column of $M$. Then
\begin{enumerate}
	\item Each entry of $M$ has absolute value at least $\mu$ and at most $1$.
	\item For each $j \in [n]$, $\|M_j\|^2 \geq p \mu^2 $.
	\item For each $j \neq j' \in [n]$, $\left|\ip{M_j}{M_{j'}}\right| \leq g$.
\end{enumerate}
\end{claim}

The main technical ingredient in the proof is the following simple  observation regarding the eigenvalues of {\em diagonal dominant} matrices, i.e., matrices in which the diagonal elements are much larger than the off-diagonal elements. This lemma can be viewed as an extension of a folklore result regarding the {\em rank} of such matrices  (see, e.g., \cite{Alo09}). The proof is a simple application of the Hoffman-Wielandt inequality.

\begin{lem}\label{lem-pert}
Let $X = (X_{ij})_{i,j \in [n]}$ be an $n \times n$ complex Hermitian matrix with  eigenvalues $ \lambda_1, \ldots,\lambda_n$. Suppose that for all $i \in [n]$ we have $X_{ii} \geq K$, where $K$ is some positive real number. Then,
$$ \left| \left\{ i \in [n]\,\,\,|\,\,\, \lambda_i \leq K/4 \right\} \right| \leq \frac{2}{K^2} \sum_{i \neq j} |X_{ij}|^2.$$
\end{lem}
\begin{proof}
Let $D$ be an $n \times n$ diagonal matrix with $D_{ii} = X_{ii}$ for all $i \in [n]$. Clearly, the eigenvalues of $D$ are $D_{11},\ldots,D_{nn}$. The Hoffman-Wielandt inequality \cite{HW53} states that, under some ordering of the eigenvalues of $X$ (w.l.o.g the one we have chosen) we have
$$ \sum_{i \in [n]} | \lambda_i - D_{ii} |^2 \leq ||X - D||^2 = \sum_{i \neq j} |X_{ij}|^2. $$
Using the fact that all $D_{ii}$'s are at least $K$ we get the required bound.
\end{proof}

Let $\sigma_1,\ldots,\sigma_n$ be the singular values of the matrix $M$ (recall that these are the square roots of the eigenvalues of the PSD matrix $M^* M$). Let $r_1,\ldots,r_n$ be the corresponding right singular vectors (i.e., the corresponding eigenvectors of $M^*M$). We thus have
\begin{enumerate}
	\item $r_1,\ldots,r_n$ form an orthonormal basis of $\C^n$.
	\item For each $j \in [n]$, $\|Mr_j\| = \sigma_j$.
	\item The vectors $ Mr_1, \ldots, Mr_n$ are orthogonal (i.e., $\ip{ Mr_i}{ Mr_j} = 0$ for $i \neq j$).
\end{enumerate}
Let $$J = \{ j \in [n] \,|\, \sigma_j \leq \mu \sqrt{p}/2 \}$$ and let $$ L = \span\{ r_j \,|\, j \in J \}.$$ We will now show that $L$ is of small dimension and that most columns of $A$ are close to $L$. We start by bounding the dimension of $L$.
\begin{claim}\label{cla-dimL}
Let $L$ be as above. Then $|J| = \dim(L) \leq \frac{2n^2 g^2}{p^2 \mu^4}$.
\end{claim}
\begin{proof}
Consider the $n \times n$ matrix $X =  M^*  M$ with eigenvalues $\sigma_1^2,\ldots,\sigma_n^2$. By Claim~\ref{cla-propM} the  diagonal elements of $X$ are all lower-bounded by $ p\mu^2 $ and the off-diagonal elements of $X$ are all upper bounded by $g$ in absolute value. Using Lemma~\ref{lem-pert}, and these bounds on the entries of $X$, we get that
$$ \left| \left\{ i \in [n] \,\,|\,\, \sigma_i^2 \leq p \mu^2/4 \right\}\right| \leq \frac{2n^2 g^2}{p^2 \mu^4}. $$ Taking square roots completes the proof.
\end{proof}

 Let $u_1,\ldots,u_n$ denote the columns of $A$. We can write each $ u_j$ in the orthonormal basis $r_1,\ldots,r_n$ in a unique way as $$  u_j = \sum_{k=1}^n \alpha_{jk} r_k.$$ Observe that 
\begin{equation}\label{eq-distuL}
	\dist( u_j, L)^2 = \sum_{k \not\in J} |\alpha_{jk}|^2
\end{equation} 

Denote the rows of the matrix $E =  M   A$ by $e_i, i \in [m]$ so that $\|e_i\| \leq \eps$ for all $i \in [m]$. Let $f_1,\ldots,f_n$ be the columns of $E$ and observe that
\begin{equation}\label{eq-sumnormf}
	\sum_{j \in [n]}\| f_j\|^2 = \sum_{i \in [m] } \|  e_i\|^2 \leq m\eps^2 
\end{equation}

The next claim bounds the sum of distances of the vectors $ u_j$ to the subspace $L$.
\begin{claim}\label{cla-sumofdist}
With the above notations, we have
$$ \sum_{j=1}^n \dist(  u_j, L)^2 \leq \frac{4m \eps^2}{\mu^2 p}.$$
\end{claim} 
\begin{proof}
Using  (\ref{eq-distuL}), (\ref{eq-sumnormf}), the orthogonality of the $ Mr_j$'s and the fact that $\sigma_j >  \frac{\mu \sqrt{p}}{2}$ for all $j \not\in J$, we have
\begin{eqnarray*}
m\eps^2 &\geq& \sum_{j \in [n]} \|f_j\|^2 = \sum_{j \in [n]} \|  M u_j \|^2 \\
&=&  \sum_{j \in [n]}\left\| \sum_{k \in [n]} \alpha_{jk}  Mr_k \right\|^2 \\
&=& \sum_{j \in [n]} \sum_{k \in [n]} |\alpha_{jk}|^2 \sigma_k^2 \\
&\geq& \frac{\mu^2 p}{4} \sum_{j \in [n]} \sum_{k \not\in J} |\alpha_{jk}|^2  \\
&=& \frac{\mu^2 p}{4} \sum_{j \in [n]} \dist( u_j, L)^2.  
\end{eqnarray*}
This proves the claim.
\end{proof}

 We now use Claim~\ref{cla-sumofdist} to deduce that many {\em rows} of $ A$ are close to a low dimensional subspace. 
\begin{claim}\label{cla-sumofdistrows}
There exists a subspace $L' \subset \C^n$ with $\dim(L') \leq \frac{2n^2 g^2}{p^2 \mu^4}$ and s.t
$$ \sum_{j=1}^n \dist(  v_j, L')^2 \leq \frac{4m \eps^2}{\mu^2 p}.$$
\end{claim}
\begin{proof}
Let $Y$ be an $n\times n$ matrix such that the $j$'th  column of $Y$ is the element of  $L$ closest to $ u_j$. If we let $L'$ be the span of the {\em rows} of $Y$ we have $\dim(L') \leq \dim(L)$ and, using Claim~\ref{cla-sumofdist}, 
$$ \sum_{j \in [n]}\dist( v_j,L')^2 \leq \| Y -  A\|^2 = \sum_{j \in [n]}\dist( u_j,L)^2  \leq \frac{4m \eps^2}{\mu^2 p}.$$
\end{proof}

This claim completes the proof of Theorem~\ref{thm-apxsggenav}.
\end{proof}

\subsection*{Proof of Theorem~\ref{thm-apxsggen} using Theorem~\ref{thm-apxsggenav}}

From Theorem~\ref{thm-apxsggenav} we can get a large subset of $V$ that is $\eps'$-close to a low dimensional subspace $L$. To derive the conclusion of Theorem~\ref{thm-apxsggen}, we will show that the rest of the points in $V$ are also close to $L$, though with a slightly worse bound on the distance. This will follow by showing that, for every point $v \in V$, there are two points $u,w \in V$ that are close to $L$ and s.t $v$ is close to the line passing through them. This will imply that $v$ is also close to $L$. The details follow.

First, apply Theorem~\ref{thm-apxsggenav} to get a subspace $L$ so that $$\dim(L) \leq \frac{2n^2 g^2}{p^2 \mu^4}$$ and such that 	$$ \sum_{i=1}^n \dist( v_i, L)^2 \leq \frac{4m \eps^2}{\mu^2 p}.$$ 

Let $$I = \left\{i \in [n] \, \left|\, \dist(v_i,L)^2 >  \frac{4gm \eps^2}{\mu^2 p^2} \right. \right\}$$ and observe that $|I| < p/g$. Our final step is to argue that the points $v_i, i \in I$ are also close to $L'$ since they are close to the span of two points $v_j,v_k$ with $j,k \not\in I$ (using the design properties of $T$).
\begin{claim}\label{cla-therest}
For each $i \in I$ there are indices $j,k \in [n] \setminus I$ such that $\{i,j,k\} \in T$. 
\end{claim}
\begin{proof}
Fix some $i \in I$. If the claim is false then every triple in $T$ that contains $i$ must have some other element in $I$. By a pigeon hole argument, there must be an element $j \in I \setminus \{i\}$ and at least $p/|I| > g$ triples containing both $i$ and $j$, contradicting the design property of $T$.
\end{proof}

We will need the following simple lemma:

\begin{lem}\label{lem-trianglesubsp}
Let $u,v,w \in \C^d$ be an $(\eps,\mu)$-dependent triple. Let $L \subset \C^d$ be a subspace with $\dist(v,L),\dist(u,L) \leq \rho$ for some $\rho >0$. Then $\dist(w,L) \leq (\eps + 2\rho)/\mu$.
\end{lem}
\begin{proof}
Let $\alpha,\beta,\gamma$ be such that 	$|\alpha|,|\beta|,|\gamma| \in [\mu,1]$ and  $\| \alpha u + \beta v + \gamma w \| \leq \eps$. Let $v',u' \in L$ be s.t $\| v - v'\|,\|u-u' \| \leq \rho$. Then
\begin{eqnarray*}
\dist(w,L) &\leq& \| w + (\alpha/\gamma)v' + (\beta/\gamma)u' \| \\
&\leq& \| w + (\alpha/\gamma)v + (\beta/\gamma)u \| + \| (\alpha/\gamma)v - (\alpha/\gamma)v'\| + \| (\beta/\gamma)u - (\beta/\gamma)u'\| \\
&\leq& \eps/|\gamma| + |\alpha/\gamma|\rho + |\beta/\gamma|\rho \\
&\leq& (\eps + 2\rho)/\mu.
\end{eqnarray*}
\end{proof}

Combining Claim~\ref{cla-therest} with Lemma~\ref{lem-trianglesubsp} we have that each $v_i, i \in [n]$ is $\eps'$ close to $L$ with $ \eps' \leq (\eps + 2\rho)/\mu$, where $\rho = \frac{2\eps \sqrt{gm}}{p\mu}$. Simplifying, we get $$\eps' \leq \frac{5\eps \sqrt{gm}}{p\mu^2}$$ as was required. This completes the proof of Theorem~\ref{thm-apxsggen}.\qed

%%%%%%%%%%%%%%%%%%%%%%%%%%%%%%%%%%%%%%%%%%%%%%%%%%%%%%%%%%%%
\section{Proof of Theorem~\ref{thm-apxsgaffine}}\label{sec-pfaffine}
%%%%%%%%%%%%%%%%%%%%%%%%%%%%%%%%%%%%%%%%%%%%%%%%%%%%%%%%%%%%
We start with some preliminary lemmas.

\begin{lem}\label{lem-triplecoefaffine}
Let $\{u,v,w\} \in \C^d$ be $B$-balanced. If $w \in \line_\eps(u,v)$ with $\eps < 1/2$ then the triple $u,v,w$ is $(\eps,1/4B)$-dependent. Furthermore, there exists a complex $\alpha$ with $|\alpha| \geq 1/4B$ such that $\| w - \alpha u - (1-\alpha)v \| \leq \eps$.
\end{lem}
\begin{proof}
By shifting $w$ to zero we can assume that both $u$ and $v$ have norm bounded by $B$.
By definition, there exists  $\alpha \in \C$ such that $\| w - \alpha u - (1-\alpha)v \| \leq \eps$ and so we only need to show that $|\alpha| \geq 1/4B$ (the same argument will apply to $1-\alpha$ by symmetry). Observe that
\begin{eqnarray*}
1 &\leq& \|w - v \| \\
&\leq& \|w - \alpha u - (1-\alpha) v \| + \|\alpha u\| + \|\alpha v\| \\
&\leq& \eps + 2\alpha B,
\end{eqnarray*}
which proves the lemma.
\end{proof}

\begin{lem}\label{lem-arcreverseaffine}
Let $\{u,v,w\} \in \C^d$ be $B$-balanced and let $0 < \eps \leq 1/2 $ be a real number such that $w \in \line_{\eps}(u,v)$. Then $v \in \line_{\eps'}(w,u)$ with $\eps' =  4\eps B$.
\end{lem}
\begin{proof}
By Lemma~\ref{lem-triplecoefaffine} there exists a complex  $\alpha$ with $|\alpha| \geq 1/4B$ such that $$ \|w - \alpha v - (1-\alpha) u \| \leq \eps.$$ Then  $$ \|v - (1/\alpha) w + (1/\alpha-1) v\| \leq \eps/\alpha \leq 4\eps B.$$ This completes the proof.
\end{proof}

\begin{lem}\label{lem-arcdensityaffine}
Let $u,v \in \C^d$ be two distinct points. Let $k$ be the maximum size of a $B$-balanced set contained in $\line_\eps(u,v)$. If $\eps < 1/4$ then $k \leq 5B$.
\end{lem}
\begin{proof}
Suppose $k > 5B$ and let $V = \{v_1,\ldots,v_k\}$ be a $B$-balanced set contained in $\line_\eps(u,v)$. For each $v_i$ let $u_i \in \line(u,v)$ be a point of distance at most $\eps$ from it. Since the $k$ points $u_1,\ldots,u_k$ are all on a line segment of length at most $2B$ we can apply a pigeon hole argument to conclude that there must be $i \neq j$ with $\dist(u_i,u_j) \leq 2B/(k-1)$. This implies $\dist(v_i,v_j) \leq 2\eps + 2B/(k-1) < 1$, a contradiction.
\end{proof}

\subsection*{Proof of Theorem~\ref{thm-apxsgaffine}}

We define $T \subset {[n] \choose 3}$ to be the set of triples $\{i,j,k\} \subset [n]$ (with three distinct indices) for which $v_k \in \line_\eps(v_i,v_j)$. By Lemma~\ref{lem-triplecoefaffine} we have that for each triple $\{i,j,k\}$ in $T$, the corresponding triple $v_i,v_j,v_k \in \C^d$ is $(\eps,1/4B)$-dependent.

\begin{claim}\label{cla-Tdesignaffine}
$T$ as defined above is a $(p,g)$ design with $p = \delta(n-1)$ and $g < 5B$.
\end{claim}
\begin{proof}
By the conditions of the theorem, each $v_i$ is contained in at least $\delta(n-1)$ triples that are in $T$ and so the bound on $p$ holds. To prove the bound on $g$, fix $i \neq j \in [n]$. If the triple $\{i,j,k\}$ appears in $T$. Then either $v_k \in \line_\eps(v_i,v_j)$, $v_i \in \line_\eps(v_j,v_k)$ or $v_j \in \line_\eps(v_i,v_k)$. In all three cases, we have, using Lemma~\ref{lem-arcreverseaffine}, that $v_k \in \line_{\eps'}(v_i,v_j)$ with $\eps' = 4\eps B$. Since $\eps < 1/16B$ we have $\eps' < 1/4$ and we can apply Lemma~\ref{lem-arcdensityaffine} to conclude that there could be at most $5B$ such triples.
\end{proof}

Observe that we can discard some of the triples in $T$ so that $|T| \leq \delta n^2$ and so that $T$ is still a $(p,g)$-design (simply keep for each $i$ only $\delta (n-1)$ dependent triples).

Plugging the  bounds obtained in the above claims and the  bound $|T| \leq \delta n^2$ into Theorem~\ref{thm-apxsggen} we get a subspace $L$ with $\dim(L) \leq O(B^6/\delta^2)$ and such that $\dist(v_i,L) \leq O(\eps B^{2.5}/\sqrt{\delta})$ for all $i \in [n]$. The second part of the theorem follows from applying Theorem~\ref{thm-apxsggenav}.

%%%%%%%%%%%%%%%%%%%%%%%%%%%%%%%%%%%%%%%%%%%%%%%%%%%%%%%%%%%%
\section{Proof of Theorem~\ref{thm-apxsg}}\label{sec-pfproj}
%%%%%%%%%%%%%%%%%%%%%%%%%%%%%%%%%%%%%%%%%%%%%%%%%%%%%%%%%%%%

We first prove some preliminary lemmas.

\begin{lem}\label{lem-distance}
Suppose $u,v \in S^d$ are s.t $\min\{\dist(u,v),\dist(u,-v)\} = \mu$. Then, for all complex $\beta$, $\dist(u,\beta v) \geq \mu/4$. 
\end{lem}
\begin{proof}
Suppose w.l.o.g $\dist(u,v) = \mu \leq \sqrt{2}$. We have $$ \mu = \sqrt{ \ip{u-v}{{u-v}}} = \sqrt{ 2 - 2\ip{u}{ v}},$$ which gives 
$ \ip{u}{ v} = 1 - \mu^2/2$. Since $\dist(u, \gamma v)$ is minimized for $\gamma = \ip{u}{ v}$ we have $\dist(u,\beta v) \geq \dist(u,(1-\mu^2/2)v) = || u -v + (\mu^2/2)v|| \geq ||u-v|| - ||(\mu^2/2)v|| \geq \mu - \mu^2/2 \geq \mu/4$ (for $\mu \leq \sqrt{2}$). 
\end{proof}

\begin{lem}\label{lem-triplecoef1}
Let $u,v,w \in S^d$ be distinct and let $\eps,\mu>0$ be real numbers s.t $\eps < \mu/8$. Suppose $\|w - \alpha u - \beta v \| \leq \eps$ for some complex numbers $\alpha,\beta$. If $\min\{ \dist(w,v),\dist(w,-v) \} \geq \mu$ then $|\alpha| > \mu/8$. 
\end{lem}
\begin{proof}
By the triangle inequality $$ \|w - \beta v \| \leq \| \alpha u\| + \eps = |\alpha| + \eps. $$ Using Lemma~\ref{lem-distance} we have $\dist(w,\beta v) \geq \mu/4$ which gives $|\alpha| \geq \mu/4 - \eps \geq \mu/8.$
\end{proof}

\begin{lem}\label{lem-triplecoef}
Let $u,v,w \in S^d$ be $\mu$-separated and suppose $\eps < \mu/8$. Suppose $w \in \arc_\eps(u,v)$. Then, there exist complex numbers $\alpha,\beta,\gamma$ with $\| \alpha u + \beta v + \gamma w \| \leq \eps$ and s.t $\mu/8 \leq |\alpha|,|\beta|,|\gamma| \leq 1$. 
\end{lem}
\begin{proof}
By the assumption, there are $\alpha',\beta'$ with $\|w - \alpha' u - \beta' v\| \leq \eps$. If $|\alpha'|$ and $|\beta'|$ are at most 1 then we are done using Lemma~\ref{lem-triplecoef1}. If not, suppose $|\alpha'| = \max\{|\alpha'|,|\beta'|\} > 1$ and divide the equation by $\alpha'$ to obtain $\|(1/\alpha')w - u - (\beta'/\alpha')v\| \leq \eps/|\alpha'| < \eps$. Now, all three coefficients are at most 1 in absolute value and, using Lemma~\ref{lem-triplecoef1}, we have the lower bound $\mu/8$ on $|1/\alpha'|, |\beta'/\alpha'|$.
\end{proof}

\begin{lem}\label{lem-arcreverse}
Let $u,v,w \in S^d$ be distinct. Let $\eps,\mu > 0$ be real numbers such that $\eps < \mu/8$. Suppose $w \in \arc_{\eps}(u,v)$ and $\min\{\dist(w,v),\dist(w,-v)\} \geq \mu$. Then $u \in \arc_{\eps'}(w,v)$ with $\eps' =  8\eps/\mu$.
\end{lem}
\begin{proof}
By our assumption, there exist complex numbers $\alpha,\beta$ such that $$ \|w - \alpha u - \beta v \| \leq \eps.$$ By Lemma~\ref{lem-triplecoef1} we have $|\alpha| > \mu/8$ and so  $$ \|u - (1/\alpha) w + (\beta/\alpha) v\| \leq 8\eps/\mu.$$ This implies $u \in \arc_{\eps'}(w,v)$ as was required.
\end{proof}

\begin{lem}\label{lem-arcdensity}
Let $u,v \in S^d$ be two distinct points. Let $k$ be the maximum size of a $\mu$-separated set contained in $\arc_\eps(u,v)$. If $\eps < \mu/4$ then $k \leq 8/\mu$.
\end{lem}
\begin{proof}
Suppose $k > 8/\mu$ and let $V = \{v_1,\ldots,v_k\}$ be a $\mu$-separated set contained in $\arc_\eps(u,v)$. For each $v_i$ let $u_i \in \arc(u,v)$ be a point of distance at most $\eps$ from it. By a pigeon hole argument, there must be $i \neq j$ with $\min\{\dist(u_i,u_j),\dist(u_i,-u_j)\} \leq \pi/k \leq \mu/2$. This implies $\min\{\dist(v_i,v_j),\dist(v_i,-v_j)\} \leq 2\eps + \mu/2 < \mu$, a contradiction.
\end{proof}

\begin{proof}[Proof of Theorem~\ref{thm-apxsg}]
	
	To reduce to Theorem~\ref{thm-apxsggen} we will define $T \subset {[n] \choose 3}$ to be the set of triples $\{i,j,k\} \subset [n]$ for which $v_k \in \arc_\eps(v_i,v_j)$. 

	\begin{claim}\label{cla-Tdepsphere}
	Let $\{i,j,k\} \in T$. Then the triple $v_i,v_j,v_k \in \C^d$ is $(\eps,\mu/8)$-dependent.
	\end{claim}
	\begin{proof}
	This is immediate from Lemma~\ref{lem-triplecoef}.
	\end{proof}

	\begin{claim}\label{cla-Tdesignsphere}
	$T$ as defined above is a $(p,g)$ design with $p = \delta(n-1)$ and $g < 8/\mu$.
	\end{claim}
	\begin{proof}
	By the conditions of the theorem, each $v_i$ is contained in at least $\delta(n-1)$ triples that are in $T$ and so the bound on $p$ holds. To prove the bound on $g$, fix $i \neq j \in [n]$. If the triple $\{i,j,k\}$ appears in $T$, then either $v_k \in \arc_\eps(v_i,v_j)$, $v_i \in \arc_\eps(v_j,v_k)$ or $v_j \in \arc_\eps(v_i,v_k)$. In all three cases, we have, using Lemma~\ref{lem-arcreverse}, that $v_k \in \arc_{\eps'}(v_i,v_j)$ with $\eps' = 8\eps/\mu$. Since $\eps < \mu^2/32$ we have $\eps' < \mu/4$ and we can apply Lemma~\ref{lem-arcdensity} to conclude that there could be at most $8/\mu$ such triples.
	\end{proof}

	Plugging the  bounds obtained in the above claims and the  bound $|T| \leq \delta n^2$ (which can be obtained by discarding some of the triples in $T$, as before) into Theorem~\ref{thm-apxsggen} and into Theorem~\ref{thm-apxsggenav} completes the proof.
\end{proof}

%%%%%%%%%%%%%%%%%%%%%%%%%%%%%%%%%%%%%%%%%%%%%%%%%%%%%%%%%%%%
\section{Proof of Theorem~\ref{thm-apxlcc}}\label{sec-apxlcc}
%%%%%%%%%%%%%%%%%%%%%%%%%%%%%%%%%%%%%%%%%%%%%%%%%%%%%%%%%%%%

Since the proof follows the same lines as the proof of Theorem~\ref{thm-apxsggen}, we will assume  familiarity with the proof of that theorem and only give details where the proofs differ.

We will use the following definition:

\begin{define}[LCC-matrix]\label{def-lccmat}	
Let $M$ be an $nk \times n$ matrix over $\C$ and let $M_1,\ldots,M_n$ be $k \times n$ matrices so that $M$ is the concatenation of the blocks $M_1,\ldots,M_n$ placed on top of each other (so $M_\ell$ contains the rows of $M$ numbered $k(\ell-1) + 1, \ldots,k \ell$). We say that $M$ is a $(k,q)$-LCC matrix if, for each $i \in [n]$ the block $M_i$ satisfies the following conditions:
\begin{itemize}
	\item Each row of $M_i$ has support size at most $q+1$.
	\item All rows in $M_i$ have the value $1$ in position $i$.
	\item The supports of two distinct rows in $M_i$ intersect only in position $i$.
\end{itemize}
\end{define}

Let $V = \{v_1,\ldots,v_n\} \subset \C^d$ be a $(q,\delta,B,\eps)$-stable LCC and assume w.l.o.g that $d=n$ (that is, pad the vectors $v_i$ with zeros so that we can think of them as vectors in $\C^n$). Let $A$ be the $n \times n$ matrix with rows $v_i$. 

\begin{claim}\label{cla-lccM}
There exists a $(k,q)$-LCC matrix $M$ with dimensions $nk \times n$ and with $k = \Omega(\delta n/q)$ such that all entries of $M$ have absolute values at most $B$ and such that $$ ||M A ||^2 \leq n^2 \eps^2.$$
\end{claim}
\begin{proof}
We will show how to construct the $k \times n$ block $M_i$ of $M$ (see Definition~\ref{def-lccmat}) row by row. Using the definition of stable LCC, there exists a family $Q_i$ of $k = \Omega(\delta n/q)$ disjoint $q$-tuples of elements of $V$ such  that, for each $q$-tuple $J \in Q_i$, we have $\dist(v_i,\span_B(J)) \leq \eps$. Each of these $q$-tuples, $J$, defines a row vector $w_J$ with $1$ in the $i$'th position, $B$-bounded entries in positions indexed by $J$, and zeros everywhere else in the following manner: Suppose $v_i = \sum_{j \in J}b_jv_j + e$ with $|b_j| \leq B$ for all $j \in J$ and $||e|| \leq \eps$. Then we define $w_j$ to have $1$ in position $i$ and values $-b_j$ in positions $j \in J$ (with zeros in all other positions).  Then, we have $||w_J A|| = ||e|| \leq \eps$. Taking all these row vectors to construct $M_i$ we get the required bound on $||MA||^2$.
\end{proof}

Let $E = MA$ so that $||E||^2 \leq n^2 \eps^2$. We now construct another $nk \times n$ matrix $R$ so that $R^T M$ will be diagonal dominant. $R$ will be comprised of $n$ blocks, $R_1,\ldots,R_n$, each of dimensions $k \times n$ so that $R_i$ has $1$'s in the $i$'th column and zeros everywhere else. Notice that, the $i$'th row of $R^T M$ is the sum of the rows in the block $M_i$ of $M$.

Let $\hat M = R^TM$ and $\hat E = R^TE$ so that $\hat E = \hat M A$. An application of the  Cauchy-Schwarz inequality shows that $$|| R^TE ||^2 \leq n ||E||^2 \leq n^3 \eps^2.$$ Observe that the diagonal elements of $\hat M$ are all equal to $k$ and that the off-diagonal elements of $\hat M$ are all of absolute value at most $B$ (since the supports of rows in $M_i$ are disjoint except for the $i$'th coordinate).
 
We proceed with analyzing the spectrum of $\hat M$. Let $r_1,\ldots,r_n$ be the right singular vectors and $\sigma_1,\ldots,\sigma_n$ the corresponding singular values. If we take $X =\hat M^* \hat M$ then the diagonal elements of $X$ are all at least $K^2 \geq k^2$ and the off diagonal elements can be bounded by $2kB + nB^2 \leq O(nB^2)$. If we define $$ L = \span\{ r_j \,|\, \sigma_j < K/2\}$$ we get that, using Lemma~\ref{lem-pert}, $$\dim(L) \leq O(n^4 B^4/K^4) = O((qB/\delta)^4).$$ 

As in the proof of Theorem~\ref{thm-apxsggen}, we consider the columns $u_1,\ldots,u_n$ of $A$ and obtain the bound
$$ \sum_{j=1}^n \dist(u_j,L)^2 \leq 4||\hat E||^2/K^2 =  O(n^3 \eps^2/K^2). $$ This means that there is a subspace $L'$ with the same dimension as $L$ such that
$$ \sum_{i=1}^n \dist(v_j,L')^2 \leq  O(n^3 \eps^2/K^2). $$
Thus, there is a set $V' \subset V$  of size $n' \geq (1 - \delta/2)n$ such that for all $ v' \in V'$ we have $\dist(v', L')^2 \leq O(n^2 \eps^2/\delta K^2) = O(q^2 \eps^2/\delta^3)$. To finish the proof we observe that, using the definition of a stable LCC, for every $v \in V$ there is a $q$-tuple $ J \subset V'$ with $\dist(v_i, \span_B(J)) \leq \eps$. Using the bound on the distances of elements of $V'$ to $L'$ and the bound $B$ on the coefficients in the linear combinations in $\span_B(J)$, we get that $\dist(v,L') \leq \eps + O( qB \cdot (q \eps/\delta^{1.5}) ) = O( q^2 B \eps / \delta^{1.5})$. This completes the proof of Theorem~\ref{thm-apxlcc}.

\bibliographystyle{alpha}

\bibliography{approxsg}
%%%%%%%%%%%%%%%%%%%%%%%%%%%%%%%%%%%%%%%%%%%%%%%%%%%

\end{document}